\documentclass[12pt]{amsart}
\usepackage{amscd,amssymb,amsthm,verbatim}
\newcommand{\C}{{\mathbb C}}

\newcommand{\cone}{\operatorname{cone}}
\newcommand{\const}{\operatorname{const.}}

\newcommand{\dbar}{\overline{\partial}}

\newcommand{\Dom}{\operatorname{Dom}}

\newcommand{\HH}{\operatorname{H}}

\newcommand{\Image}{\operatorname{Im}}

\newcommand{\Index}{\operatorname{Index}}

\newcommand{\KK}{\operatorname{K}}
\newcommand{\KKK}{\operatorname{KK}}
\newcommand{\Ker}{\operatorname{Ker}}

\newcommand{\pt}{\operatorname{pt}}
\newcommand{\Q}{{\mathbb Q}}

\newcommand{\Td}{\operatorname{Td}}

\newcommand{\Z}{{\mathbb Z}}

\numberwithin{equation}{section}
\theoremstyle{plain}

\newtheorem{lemma}[equation]{Lemma}
\newtheorem{theorem}[equation]{Theorem}
\newtheorem{proposition}[equation]{Proposition}

\theoremstyle{remark}

\newtheorem{remark}[equation]{Remark}

\usepackage{graphicx}
\input{amssym.def}
\input{amssym}
\input{epsf}
\usepackage{a4wide}

\begin{document}

\title{A Dolbeault-Hilbert complex for a variety with
isolated singular points}

\author{John Lott}
\address{Department of Mathematics\\
University of California, Berkeley\\
Berkeley, CA  94720-3840\\
USA} \email{lott@berkeley.edu}

\thanks{Research partially supported by NSF grant
DMS-1810700}
\date{July 11, 2019}

\begin{abstract}
  Given a compact Hermitian complex space with isolated singular points,
  we construct a Dolbeault-type Hilbert complex whose cohomology is isomorphic to the cohomology of the structure sheaf. We show that the
  corresponding
  K-homology class coincides with the one constructed by
  Baum-Fulton-MacPherson.
\end{abstract}

\maketitle

\section{Introduction} \label{section1}

The program
of doing index theory, or more generally elliptic theory, on
singular varieties goes back at least to Singer's paper
\cite[\S 4]{Singer (1971)}. This program
takes various directions, for example the
relation between $L^2$-cohomology and intersection
homology. In this paper we consider a somewhat different direction, which is
related to the arithmetic genus. This is
motivated by work of Baum-Fulton-MacPherson
\cite{Baum-Fulton-MacPherson (1975),Baum-Fulton-MacPherson (1979)}. 

Let $X$ be a projective complex algebraic variety and let ${\mathcal S}$ be a
coherent sheaf on $X$. 
In \cite{Baum-Fulton-MacPherson (1979)}, the authors associated to
${\mathcal S}$ an element
$[{\mathcal S}]_{BFM} \in \KK_0(X)$ of the topological K-homology of X.
This class enters into their Riemann-Roch theorem for singular varieties.
In
particular, under the map $p \: : \: X \rightarrow \pt$, the image
$p_* [{\mathcal S}]_{BFM} \in \KK_0(\pt) \cong \Z$ is expressed in terms of 
sheaf cohomology by
$\sum_i (-1)^i \dim(\HH^i(X; {\mathcal S}))$.

In view of the isomorphism between topological K-homology and analytic
K-homology
\cite{Baum-Douglas (1982),Baum-Higson-Schick (2007)}, the class 
$[{\mathcal S}]_{BFM}$ can be represented by an ``abstract elliptic operator''
in the sense of Atiyah \cite{Atiyah (1970)}.
This raised
the question of how to find an explicit
cycle in analytic K-homology, even if $X$ is singular,
that represents $[{\mathcal S}]_{BFM}$.
The most basic case is when ${\mathcal S}$ is the structure sheaf 
${\mathcal O}_X$.
If $X$ is smooth then the
operator representing
$[{\mathcal O}_X]_{BFM}$
is $\dbar + \dbar^*$. Hence we are looking for the right analog of this
operator when $X$ may be singular.

A second related question is to
find a Hilbert complex, in the sense of \cite{Bruning-Lesch (1992)},
whose cohomology is isomorphic to
$\HH^\star(X; {\mathcal O}_X)$. We want
the complex to be intrinsic to $X$. Also, if $X$ is smooth then we want to
recover the $\dbar$-complex on $(0, \star)$-forms.

In this paper, we answer these
questions when $X$ has isolated singular points.
To see the nature of the problem, suppose that $X$ is a complex
curve, whose normalization has genus $g$.
In this case, the Riemann-Roch theorem says
\begin{equation} \label{1.1}
  \dim(\HH^0(X; {\mathcal O}_X)) - \dim(\HH^1(X; {\mathcal O}_X)) =
  1 - g - \sum_{x \in X_{sing}} \delta_x,
\end{equation}
where $\delta_x$ is a certain positive integer attached to the
singular point $x$
\cite[p. 298]{Hartshorne (1997)}.
To find the appropriate Hilbert complex, it is
natural to start with the Dolbeault complex
$\Omega^{0,0}_c(X_{reg}) \stackrel{\dbar}{\longrightarrow}
\Omega^{0,1}_c(X_{reg})$ of smooth compactly supported forms on
$X_{reg}$ and look for a  closed operator extension, where
$X_{reg}$ is endowed with the induced Riemannian metric from its
projective embedding. For the minimal closure
$\dbar_s$, one finds $\Index(\dbar_s) = 1-g$.  Taking a different
closure can only make the index go up
\cite{Bruning-Peyerimhoff-Schroder (1990)},
whereas in view of (\ref{1.1}) we want the index to go down.
(Considering complete Riemannian metrics on $X_{reg}$ does not help.)
However,
on the level of indices, we can get the right answer by enhancing the
codomain by $\bigoplus_{x \in X_{sing}} \C^{\delta_x}$.

Now let $X$ be a compact Hermitian complex space of pure dimension $n$.
For technical reasons, we assume that the singular set $X_{sing}$
consists of isolated singularities.  (In the bulk of the paper we allow
coupling to a holomorphic vector bundle, but in this introduction we only
discuss the case when the vector bundle is trivial.)  Let $\dbar_s$ be the
minimal closed extension of the $\dbar$-operator on
$X_{reg} = X - X_{sing}$. Its domain
${\Dom}(\dbar_s^{0,\star})$ can be localized to a complex of sheaves
$\underline{\Dom}(\dbar_s^{0,\star})$. Let
$\underline{\HH}^{0,\star}(\dbar_s)$ denote the cohomology, a
sum of skyscraper sheaves on $X$ if $\star > 0$.
We write ${\mathcal O}_s$ for
$\underline{\HH}^{0,0}(\dbar_s)$, which is the sheaf of
germs of weakly holomorphic
functions on $X$, the latter being in the sense of
\cite[Section 4.3]{Whitney (1972)}.
Then ${\mathcal O}_s/{\mathcal O}_X$ is also a sum of
skyscraper sheaves on $X$.
Its vector space of global sections will be written as
$({\mathcal O}_s/{\mathcal O}_X)(X)$.
Both
$\underline{\HH}^{0,\star}(\dbar_s)$ and
${\mathcal O}_s/{\mathcal O}_X$ can be computed using a resolution
of $X$ \cite[Corollary 1.2]{Ruppenthal (2018)}.

Define vector spaces $T^*$ by
\begin{align} \label{1.2}
    T^0 = & {\Dom}(\dbar_s^{0,0}), \\
  T^1 = & {\Dom}(\dbar_s^{0,1}) \oplus
  ({\mathcal O}_s/{\mathcal O}_X)(X), \notag \\
    T^\star = & {\Dom}(\dbar_s^{0,\star}) \oplus
    (\underline{\HH}^{0,\star-1}(\dbar_s))(X), \text{ if } 2 \le \star \le n.
    \notag  
  \end{align}

  To define a differential on $T^*$, let $\triangle_s^{0,\star}$ be the Laplacian
  associated to $\dbar_s$. Let $P_{\Ker(\triangle_s^{0,\star})}$ be
  orthogonal projection onto the kernel of $\triangle_s^{0,\star}$.
  As the elements of $\Ker(\triangle_s^{0,\star})$ are
  $\dbar_s$-closed, for each $x \in X_{sing}$ there is a well-defined
  map $\Ker(\triangle_s^{0,\star}) \rightarrow
  (\underline{\HH}^{0,\star}(\dbar_s))_x$ to the stalk of
  $\underline{\HH}^{0,\star}(\dbar_s)$ at $x$.
  For $\star > 0$, putting these together
  for all $x \in X_{sing}$, and precomposing with
  $P_{\Ker(\triangle_s^{0,\star})}$, gives a linear map
  $\gamma : {\Dom}(\dbar_s^{0,\star}) \rightarrow
  (\underline{\HH}^{0,\star}(\dbar_s))(X)$.
  For $\star = 0$, we similarly define
  $\gamma : {\Dom}(\dbar_s^{0,0}) \rightarrow
  ({\mathcal O}_s/{\mathcal O}_X)(X)$.
Define a differential $d : T^* \rightarrow T^{*+1}$ by
\begin{align} \label{1.3}  
  d(\omega) = & (\dbar_s \omega, \gamma(\omega)), \text{ if } \star = 0, \\
  d(\omega, a) = & ( \dbar_s \omega, \gamma(\omega)),
  \text{ if } \star > 0. \notag 
\end{align}  

\begin{theorem} \label{1.4}
  The cohomology of $(T, d)$ is isomorphic to $\HH^*(X; {\mathcal O}_X)$.
  \end{theorem}

Theorem \ref{1.4} can be seen as an extension of
\cite[Corollary 1.3]{Ruppenthal (2018)} by Ruppenthal,
which implies the result when
$X$ is normal and has rational singularities. To prove Theorem \ref{1.4}, we
construct a certain
resolution of ${\mathcal O}_X$ by fine sheaves.  The cohomology
of the complex $(\widetilde{T}, \widetilde{d})$ of global sections is then
isomorphic to $\HH^*(X; {\mathcal O}_X)$. The complex
$(\widetilde{T}, \widetilde{d})$ is not quite the same as
$(T, d)$ but we show that they are cochain-equivalent, from which the
theorem follows.

The spectral triple $(C(X), T, d+d^*)$ defines an element
$[\mathcal O_X]_{an} \in \KK_0(X)$ of the analytic K-homology of $X$.

\begin{theorem} \label{1.5}
  If $X$ is a projective algebraic variety with isolated singularities then
  $[\mathcal O_X]_{an} = [\mathcal O_X]_{BFM}$ in
  $\KK_0(X)$.
  \end{theorem}

There has been interesting earlier work on the questions addressed in this
paper.  In \cite{Ancona-Gaveau (1994)}, Ancona and Gaveau
gave a resolution of the structure sheaf of a normal complex space $X$,
assuming that the singular locus is smooth, in terms of differential forms on
a resolution of $X$.  The construction depended on the
choice of resolution.
In \cite{Fox-Haskell (2000)}, Fox and Haskell discussed using a perturbed
Dolbeault operator on an ambient manifold to represent the K-homology class
of the structure sheaf.  In \cite{Andersson-Samuelsson (2012)},
Andersson and Samuelsson gave a resolution of the structure sheaf by certain
currents on $X$, that are smooth on $X_{reg}$. After this paper was
written, Bei and Piazza posted \cite{Bei-Piazza}, which also has a proof of
Proposition \ref{sclass}.

The structure of the paper is the following.  In Section
\ref{section2}, given a holomorphic vector bundle $V$ on $X$, we recall the
definition of the minimal closure $\dbar_{V,s}$ and show that
$\dbar_{V,s} + \dbar_{V,s}^*$ gives an element of the analytic K-homology group
$\KK_0(X)$, in the unbounded formalism for the Kasparov KK-group
$\KKK(C(X); \C)$. In Section \ref{section3} we construct
a resolution of the sheaf $\underline{V}$ by fine sheaves.
Their global sections give
a Hilbert complex.
In Section \ref{section4} we deform this to the complex $(T_V, d_V)$.
Section \ref{section5}
has the proof of Theorem \ref{1.5}.
More detailed descriptions appear at the beginning of the sections.

I thank Paul Baum and Peter Haskell for discussions.
I especially thank Peter for pointing out the relevance of
\cite{Haskell (1987)}. I also thank the referee for helpful comments.

\section{Minimal closure and compact resolvent} \label{section2}

In this section we consider a holomorphic vector bundle $V$
on a compact complex
space $X$ with isolated singularities.  We define the minimal closure
$\dbar_{V,s}$. We show that
the spectral triple
$(C(X), \overline{\partial}_{V,s} + \overline{\partial}_{V,s}^*,
\Omega_{L^2}^{0,*}(X_{reg}; V))$ gives a well-defined element
of the analytic $K$-homology group $\KK_0(X)$, in the unbounded formalism.
The main issue is to show that
$\overline{\partial}_{V,s} + \overline{\partial}_{V,s}^*$ has a
compact resolvent. When $V$ is trivial, this was shown in
\cite{Ovrelid-Ruppenthal (2014)}.

Let $X$ be a reduced compact complex space of pure dimension $n$.
 For each $x \in X$, there is
a  neighborhood $U$ of $x$
  with an embedding of $U$ into some domain $U^\prime \subset
    \C^N$, as the zero set of a finite number of holomorphic functions
    on $U^\prime$.

Let ${\mathcal O}_X$ be the analytic structure sheaf of $X$.
Let $X_{sing}$ be the set of
singular points of $X$ and
put $X_{reg} = X - X_{sing}$.

We equip $X$ with a 
Hermitian metric $g$ on $X_{reg}$ that satisfies
the following property: For each $x \in X$, there are $U$ and $U^\prime$
as above, along with
a smooth Hermitian metric $G$ on $U^\prime$, so that
    $g \Big|_{X_{reg} \cap U} = G \Big|_{X_{reg}\cap U}$.

Let $V$ be a finite dimensional
holomorphic vector bundle on $X$ or, equivalently, a locally free
sheaf $\underline{V}$
of ${\mathcal O}_X$-modules.
For each $x \in X$, there are $U$ and $U^\prime$
as above so that $V \Big|_U$ is the restriction of a trivial holomorphic bundle
$U^\prime \times \C^N$ on $U^\prime$.
Let $h$ be a Hermitian inner product on $V \Big|_{X_{reg}}$ that satisfies the
following property: For each $x \in X$, there are such $U$ and $U^\prime$
so that
$h \Big|_{X_{reg} \cap U}$ is the restriction of a
smooth Hermitian metric on $U^\prime \times \C^N$.

Let $\overline{\partial}_{V,s}$ be the minimal closed extension of the
$\overline{\partial}_V$-operator on $X_{reg}$. That is, the domain of
$\overline{\partial}_{V,s}$ is the set of
$\omega \in \Omega_{L^2}^{0,*}(X_{reg}; V)$
so that there are a sequence of compactly supported smooth forms
$\omega_i \in \Omega^{0,*}(X_{reg}; V)$ on $X_{reg}$
and some $\eta \in
\Omega_{L^2}^{0,*+1}(X_{reg}; V)$ such
that $\lim_{i \rightarrow \infty} \omega_i = \omega$ in
$\Omega_{L^2}^{0,*}(X_{reg}; V)$, and
$\lim_{i \rightarrow \infty} \dbar_{V,s} \omega_i = \eta$ in
$\Omega_{L^2}^{0,*+1}(X_{reg}; V)$. We then put $\dbar_{V,s} \omega = \eta$,
which is uniquely defined.

Hereafter we assume that $X_{sing}$ is finite.

\begin{proposition} \label{2.1}
The spectral triple
$(C(X), \overline{\partial}_{V,s} + \overline{\partial}_{V,s}^*,
\Omega_{L^2}^{0,*}(X_{reg}; V))$ gives a well-defined element
of the analytic $K$-homology group $\KK_0(X)$.
\end{proposition}
\begin{proof}
  Put $D_V = \overline{\partial}_{V,s} + \overline{\partial}_{V,s}^*$, with
  dense domain $\Dom( \dbar_{V,s}) \cap \Dom( \dbar_{V,s}^*)$.
  Put $D = \overline{\partial}_{s} + \overline{\partial}_{s}^*$, the
  case when $V$ is the trivial complex line bundle.
  Put

  \begin{equation} \label{2.2}
{\mathcal A} = \{ f \in C(X) \: : \:  
f(\Dom(D_V)) \subset \Dom(D_V) \text{ and } [D_V,f] \text{ is bounded.} \}
  \end{equation}
Using the local trivializations of $V$, it follows that
\begin{equation} \label{2.3}
  {\mathcal A} = \{ f \in C(X) \: : \:  
f(\Dom(D)) \subset \Dom(D) \text{ and } [D,f] \text{ is bounded.} \}
  \end{equation}
To satisfy the definitions of 
  unbounded analytic $K$-homology
  \cite{Baaj-Julg (1983),Forsyth-Mesland-Rennie (2014),Kaad (2019)},
  we first need to show
  that ${\mathcal A}$ is dense in $C(X)$.

  Given $F \in C(X)$ and $\epsilon > 0$,
  we can construct $f \in C(X)$ so that
  \begin{itemize}
  \item
 For each
 $x_j \in X_{sing}$, there is a neighborhood $U_j \subset X$ of
 $x_j$ on which $f$ is constant, with $f(x_j) = F(x_j)$.
\item $f$ is smooth on $X_{reg}$.
  \item $\sup_{x \in X} |f(x) - F(x)| < \epsilon$.
    \end{itemize}
Then $f(\Dom(D)) \subset \Dom(D)$  and
  $\| [D,f] \| \le \const \| \nabla_h f \|_\infty < \infty$.
It follows that
  ${\mathcal A}$ is dense in $C(X)$.

  To prove the proposition, it now suffices to prove the following lemma.
  \begin{lemma} \label{2.4}
    $(D_V+i)^{-1}$ is compact
    \end{lemma}
  \begin{proof}
    If $V$ is trivial then the lemma is true
    \cite{Ovrelid-Ruppenthal (2014)}. We will use a parametrix
    construction to prove it for general $V$.
    
    We first prove the lemma for a special inner product $h^\prime$ on $V$.
    Write $X_{sing} = \{x_j\}_{j=1}^r$. For each $j$, let $U_j$ be a
    neighborhood of $x_j$ on which $V$ is trivialized as above, with
    $\overline{U_j} \cap \overline{U_k} = \emptyset$ for $j \neq k$. 
    Choose open sets with smooth boundary
    $x_j \in Z_j \subset Y_j \subset W_j \subset U_j$, with
    $\overline{Z_j} \subset Y_j$,
    $\overline{Y_j} \subset W_j$ and
    $\overline{W_j} \subset U_j$.
    Let $\phi_j \in C(X)$ be identically one on $Y_j$, with
    support in $W_j$, and smooth on $U_j - Y_j$. 
    Let $\eta_j \in C(X)$ be identically one on $W_j$, with
    support in $U_j$, and smooth on $U_j - Y_j$, so that
    $\eta_j$ is one on the support of $\phi_j$.
    
    Define an inner product $h^\prime$ on $V$ by first taking it to be a
    trivial inner product on each $U_j$,
    in terms of our given trivializations, and then
    extending it smoothly to the rest of
    $X_{reg}$. Let $V_j$ be the extension of the trivialization
    $U_j \times \C^N$ to a
    product bundle on $X \times \C^N$ on $X$,
    as a smooth vector bundle with trivial
    inner product.  Let $D_{V_j} = D \otimes I_N$ be the
    corresponding operator.  
    As $(D+i)^{-1}$ is compact \cite{Ovrelid-Ruppenthal (2014)}, the
    same is true for $D_{V_j}$. Let $D_{APS}$ be the
    operator $\dbar_{V} + \dbar_V^*$ on
    $X - \bigcup_j Z_j$, with Atiyah-Patodi-Singer boundary conditions
    \cite{Atiyah-Patodi-Singer (1973)}.
    (The paper \cite{Atiyah-Patodi-Singer (1973)} assumes a product structure
    near the boundary, but this is not necessary.)
    Then $(D_{APS} + i)^{-1}$ is compact.
    Put $\phi_0 = 1 - \sum_j \phi_j$, with support in
    $X - \bigcup_j \overline{Z_j}$. Pick $\eta_0 \in C(X)$ with support in
    $X - \bigcup_j \overline{Z_j}$, and smooth on $X_{reg}$, such that
    $\eta_0$ is one on the support of $\phi_0$.

For $\omega \in \Omega_{L^2}^{0,*}(X_{reg}; V)$, put
\begin{equation} \label{2.5}
  Q \omega =  \eta_0 (D_{APS}+i)^{-1}  (\phi_0 \omega) +
  \sum_j \eta_j (D_{V_j}+i)^{-1}  (\phi_j \omega).
\end{equation}
Then $Q$ is compact and
\begin{equation} \label{2.6}
  (D_V + i) Q \omega = \omega + [D, \eta_0] (D_{APS}+i)^{-1} (\phi_0 \omega) +
  \sum_j [D, \eta_j] (D_{V_j}+i)^{-1} (\phi_j \omega),
\end{equation}
so
\begin{equation} \label{2.7}
  (D_V + i)^{-1} = Q - (D_V + i)^{-1}
  \left(  [D,\eta_0] (D_{APS}+i)^{-1} \phi_0  +
  \sum_j [D, \eta_j] (D_{V_j}+i)^{-1} \phi_j  \right) . 
\end{equation}
As $[D, \eta_0]$, $[D, \eta_j]$ and $(D_V + i)^{-1}$ are bounded,
it follows that
$(D_V + i)^{-1}$ is compact.

As $(D_V + i)^{-1}$ (for the inner product $h^\prime$) is compact,
the spectral theorem for compact operators and the functional calculus
imply that $(I + D_V^2)^{-1}$ is compact.  Writing 
$\triangle_{V,s} = D_V^2$,
there is then a Hodge
    decomposition
    \begin{equation} \label{2.8}
      \Omega_{L^2}^{0,*}(X_{reg}; V) = 
      \Ker(\triangle_{V,s}^{0,\star}) \oplus \Image(\dbar_{V,s}) \oplus
      \Image(\dbar_{V,s}^*)
    \end{equation}
    where the right-hand side is a sum of orthogonal closed subspaces.
    In particular,
\begin{enumerate}
\item    $\Image(\dbar_{V,s})$ is closed,
\item    $\Ker(\dbar_{V,s})/\Image(\dbar_{V,s})$ is finite dimensional and
\item    The map $\dbar_{V,s} \: : \: \Omega_{L^2}^{0,*}(X_{reg}; V)/
    \Ker(\dbar_{V,s}) \rightarrow \Image(\dbar_{V,s})$ is invertible and the
    inverse is compact, i.e. sends bounded sets to precompact sets.
\end{enumerate}
(The inverse map $\Image(\dbar_{V,s})
\rightarrow \Omega_{L^2}^{0,*}(X_{reg}; V)/
\Ker(\dbar_{V,s}) \cong \Image(\dbar_{V,s}^*)$
is $DG$, where $G$ is the Green's operator for
$\triangle_{V,s}$.)
    As the $L^2$-inner products on
    $\Omega_{L^2}^{0,*}(X_{reg}; V)$ coming from $h^\prime$ and $h$ are
    relatively bounded, the above three properties also hold for $h$.
    It follows that there is a Hodge decomposition relative to the
    inner product $h$, and
    $(I+D_V^2)^{-1}$ is compact.
    Hence $(D_V + i)^{-1}$ is compact.
\end{proof}
  This proves the proposition.
\end{proof}

\section{Resolution} \label{section3}

In this section we construct a certain resolution of the sheaf of holomorphic
sections of a holomorphic vector bundle $V$ on $X$.
To begin,
we define a sheaf
$\underline{\Dom}(\overline{\partial}_{V,s}^{0,\star})$ on $X$,
following \cite[Section 2.1]{Ruppenthal (2018)}. 

Given an open set $U \subset X$ and a compact subset $K \subset U$, we write
$U_{reg}$ for $U \cap X_{reg}$ and $K_{reg}$ for $K \cap X_{reg}$.

Let $V$ be a finite dimensional
holomorphic vector bundle on $X$
equipped with a Hermitian metric, in the sense of Section \ref{section2}.
There is a sheaf $\underline{\Omega}_{V,L^2_{loc}}^{0,\star}$ on $X$
whose sections over an open set $U \subset X$ are the
locally square integrable $V$-valued
forms of degree $(0, \star)$ on $U_{reg}$,
i.e. they are square integrable on $K_{reg}$ for any
compact set $K \subset U$. Convergence will mean $L^2$-convergence
on each such $K_{reg}$.
By definition, the sections of
$\underline{\Dom}(\overline{\partial}_{V,s}^{0,\star})$ over $U$ are
the elements $\omega \in {\Omega}_{L^2_{loc}}^{0,\star}(U_{reg}; V)$ so that
there are
\begin{itemize}
  \item A sequence $f_i \in {\Omega}_{C^\infty_c}^{0,\star}(U_{reg}; V)$
    and
  \item
    Some $\eta \in {\Omega}_{L^2_{loc}}^{0,\star+1}(U_{reg}; V)$
\end{itemize}
such
that for any compact $K \subset U$, we have
\begin{itemize}
\item   $\lim_{i \rightarrow \infty} f_i = \omega $ in
  ${\Omega}_{L^2}^{0,\star}(K_{reg}; V)$ and
\item  $\lim_{i \rightarrow \infty} \overline{\partial}_V f_i =
  \eta $ in
  ${\Omega}_{L^2}^{0,\star+1}(K_{reg}; V)$.
\end{itemize}
Then we put
  $\dbar_V \omega = \eta$.
  
  This gives a complex of fine sheaves
  \begin{equation} \label{3.1}
    \ldots \stackrel{\dbar_V}{\longrightarrow}
    \underline{\Dom}(\overline{\partial}_{V,s}^{0,\star-1})
    \stackrel{\dbar_V}{\longrightarrow}
    \underline{\Dom}(\overline{\partial}_{V,s}^{0,\star})
    \stackrel{\dbar_V}{\longrightarrow}
    \underline{\Dom}(\overline{\partial}_{V,s}^{0,\star+1})
    \stackrel{\dbar_V}{\longrightarrow}
    \ldots
  \end{equation}
  The cohomology of the complex is the sheaf
  $\underline{\HH}^{0,\star}(\overline{\partial}_{V,s})$.
  For $\star > 0$,
it is a direct sum of skyscraper sheaves, with support in $X_{sing}$.
  We write $\underline{V}_s$ for
  $\underline{\HH}^{0,0}(\overline{\partial}_{V,s})$, i.e. the kernel of
  $\dbar_V$ acting on 
  $\underline{\Dom}(\overline{\partial}_{V,s}^{0,0})$. Then
  $\underline{V}_s/\underline{V}$ is also
  a direct sum of skyscraper sheaves with support in $X_{sing}$.

  Although we will not need it here, there is a description of these
  skyscraper sheaves in terms of a resolution of $X$.
  Suppose that $\pi \: : \: M \rightarrow X$ is a resolution.
  From \cite[Corollary 1.2]{Ruppenthal (2018)}, if $x \in X$ then
  we can identify the stalk
  $(\underline{\HH}^{0,q}(\overline{\partial}_{V,s}))_x$ with
  $V_x \otimes (R^q \pi_* {\mathcal O}_M)_x$. In particular,
  we can identify 
  $\underline{V}_s$ with
  $\underline{V} \otimes_{{\mathcal O}_X} \pi_* {\mathcal O}_M$ or,
  more intrinsically, with the sheaf of weakly holomorphic sections of $V$,
  i.e. bounded holomorphic sections of
  $V \Big|_{X_{reg}}$.

  There is a quotient morphism of sheaves:
  $q \: : \: \underline{\Ker}(\overline{\partial}_{V,s}^{0,\star}) \rightarrow
  \underline{\HH}^{0,\star}(\overline{\partial}_{V,s})$.
  As $\underline{\HH}^{0,\star}(\overline{\partial}_{V,s})$
  is an injective sheaf
  for $\star > 0$,
  we can extend $q$ to a morphism
  $\alpha \: : \: \underline{\Dom}(\overline{\partial}_{V,s}^{0,\star})
  \rightarrow
  \underline{\HH}^{0,\star}(\overline{\partial}_{V,s})$.
  More specifically, if $x$ is a singular point then the stalk
  $(\underline{\HH}^{0,\star}(\overline{\partial}_{V,s}))_x$
  is a finite dimensional
  complex vector space, so we are extending the quotient map
  $q_x \: : \: (\underline{\Ker}(\overline{\partial}_{V,s}^{0,\star}))_x
  \rightarrow
  (\underline{\HH}^{0,\star}(\overline{\partial}_{V,s}))_x$ from the
  germs of $\dbar_V$-closed $V$-valued forms at $x$, to the germs of
  forms in the domain of $\dbar_{V,s}$.

  Considering $\underline{\HH}^{0,\star}(\overline{\partial}_{V,s})$ to be a
  complex of sheaves with zero differential,
  $\alpha$ is a morphism of complexes that is an
  isomorphism on cohomology in degree $\star > 0$ by construction.
  Let  $\underline{\cone}(\alpha_V)$ be the mapping cone of $\alpha_V$,
  with $\underline{\cone}^{0,\star}(\alpha_V) =
  \underline{\Dom}(\overline{\partial}_{V,s}^{0,\star}) \oplus
  \underline{\HH}^{0,\star-1}(\overline{\partial}_{V,s})$ and differential
  $d_{\cone} (\omega, h) = (\dbar_V \omega, \alpha_V(\omega))$.
  It has vanishing cohomology in degree $\star > 1$.
  Define a complex of sheaves $\underline{\mathcal C}_V^{0,\star}$ by
    \begin{equation} \label{3.2}
    \underline{\mathcal C}_V^{0,\star} =
    \begin{cases}
    \underline{\Dom}(\overline{\partial}_{V,s}^{0,0}), & \star = 0 \\
    \underline{\Dom}(\overline{\partial}_{V,s}^{0,1}), & \star = 1 \\
    \underline{\Dom}(\overline{\partial}_{V,s}^{0,\star}) \oplus
  \underline{\HH}^{0,\star-1}(\overline{\partial}_{V,s}), & \star > 1
      \end{cases}
  \end{equation}
  where the differential in degree $\star = 0$ is $\dbar_V$,
  the differential in degree $\star = 1$ is $(\dbar_V, \alpha_V)$, and
  the differential in degrees $\star > 1$ is $d_{\cone}$.
  Then $\underline{\mathcal C}_V$ is a resolution of
  $\underline{V}_s$ by fine sheaves.

  There is a short exact sequence of sheaves
  \begin{equation} \label{3.3}
    0 \longrightarrow \underline{V} \longrightarrow
    \underline{V}_s \longrightarrow \underline{V}_s/\underline{V}
    \longrightarrow 0.
    \end{equation}
  We can think of $\underline{V}_s/\underline{V}$ as a resolution of itself, when
  concentrated in degree zero. Together with the resolution of
  $\underline{V}_s$ from (\ref{3.2}), we can construct a resolution of
  $\underline{V}$ as follows.
  As $\underline{V}_s/\underline{V}$ is a finite sum of skyscraper sheaves,
  we can extend the quotient map
  $\underline{V}_s \rightarrow \underline{V}_s/\underline{V}$ to a
  morphism
  $\beta_V \: : \:  \underline{\Dom}(\overline{\partial}_{V,s}^{0,0})
  \rightarrow \underline{V}_s/\underline{V}$.
  Define a complex of sheaves $\underline{\widetilde{\mathcal C}}_V$ by
  \begin{equation} \label{3.4}
    \underline{\widetilde{\mathcal C}}_V^{0,\star} =
    \begin{cases}
    \underline{\Dom}(\overline{\partial}_{V,s}^{0,0}), & \star = 0 \\
    \underline{\Dom}(\overline{\partial}_{V,s}^{0,1})
    \oplus \underline{V}_s/\underline{V}, & \star = 1 \\
    \underline{\Dom}(\overline{\partial}_{V,s}^{0,\star}) \oplus
  \underline{\HH}^{0,\star-1}(\overline{\partial}_{V,s}), & \star > 1
      \end{cases}
    \end{equation}
  where the differential in degree $\star = 0$ is $(\dbar_V, \beta_V)$,
  the differential in degree $\star = 1$ sends
  $(\omega, v)$ to $(\dbar_V \omega, \alpha_V(\omega))$, and
  the differential in degrees $\star > 1$ is $d_{\cone}$. Then
    $\underline{\widetilde{\mathcal C}}_V$ is a resolution of
    $\underline{V}$ by fine sheaves; c.f.
    \cite[Pf. of Proposition I.6.10]{Iversen (1986)}.

    Taking global sections of $\widetilde{\mathcal C}^{0,\star}_V$
    gives a cochain complex $(\widetilde{T}_V, \widetilde{d}_V)$:
    \begin{align} \label{3.5}
&       0 \rightarrow {\Dom}(\overline{\partial}_{V,s}^{0,0})
      \rightarrow {\Dom}(\overline{\partial}_{V,s}^{0,1})
      \oplus (\underline{V}_s/\underline{V})(X) \rightarrow
      {\Dom}(\overline{\partial}_{V,s}^{0,2}) \oplus
      (\underline{\HH}^{0,1}(\overline{\partial}_{V,s}))(X)
      \rightarrow \\
      & \ldots \rightarrow
      {\Dom}(\overline{\partial}_{V,s}^{0,n}) \oplus
      (\underline{\HH}^{0,n-1}(\overline{\partial}_{V,s}))(X)
            \rightarrow 0. \notag
    \end{align}
    For the last term,
    we use the fact that in terms of a resolution $\pi \: : \:
    M \rightarrow X$, we have
    $(\underline{\HH}^{0,n}(\overline{\partial}_{V,s}))_x =
    V_x \otimes (R^n \pi_* {\mathcal O}_M)_x = 0$.

    \begin{proposition} \label{3.6}
      The cohomology of
      $(\widetilde{T}_V, \widetilde{d}_V)$
      is isomorphic to $\HH^{*}(X; \underline{V})$.
      \end{proposition}
    \begin{proof}
      This holds because
      $\underline{\widetilde{\mathcal C}}_V$ is a resolution of
    $\underline{V}$ by fine sheaves.
      \end{proof}

    Put arbitrary inner products on the finite dimensional vector spaces
    $(\underline{V}_s/\underline{V})(X)$ and
    $(\underline{\HH}^{0,*}(\overline{\partial}_{V,s}))(X)$.
    
    \section{Hilbert complex} \label{section4}

    The differential $\widetilde{d}_V$ in
    the Hilbert complex
    $(\widetilde{T}_V, \widetilde{d}_V)$ of the previous section
    involved somewhat arbitrary choices of
    $\alpha_V$ and $\beta_V$. 
    In this section we replace 
    $(\widetilde{T}_V, \widetilde{d}_V)$
    by a more canonical
    Hilbert complex $(T_V, d_V)$.
    
    For brevity of notation, we put
    \begin{equation} \label{4.1}
      A_V^* = \begin{cases}
        (\underline{V}_s/\underline{V})(X), & \star = 0 \\
        (\underline{\HH}^{0,\star}(\dbar_{V,s}))(X), & \star > 0.
        \end{cases}
    \end{equation}
    Then the complex $\widetilde{T}_V$ has entries
    $\widetilde{T}^{0,\star}_V =
    \Dom( \dbar^{0,\star}_{V,s}) \oplus  A_V^{\star-1}$
    Combining $\alpha_V$ and $\beta_V$, we have constructed a linear map
    $\gamma_V : \Dom( \dbar^{0,\star}_{V,s}) \rightarrow A_V^\star$ so that
    the differential of $\widetilde{T}_V$ is given by
\begin{equation} \label{4.2}
  \widetilde{d}_V(\omega, a) = (\partial_V \omega, \gamma_V(\omega)).
\end{equation}
Note that
    $\gamma_V \circ \dbar_{V,s} = 0$.

    Let $P_{\Ker(\triangle_{V,s}^{0,\star})}$ be orthogonal projection onto
    $\Ker(\triangle_{V,s}^{0,\star}) \subset
    \Omega^{0,\star}_{L^2}(X_{reg}; V)$.
    Define a new differential $d_V$ on $\widetilde{T}_V$ by
\begin{equation} \label{4.3}
    d_V(\omega, a) = (\partial_V \omega,
    \gamma_V(P_{\Ker(\triangle_{V,s}^{0,\star})} \omega)).
    \end{equation}
Call the resulting cochain complex $(T_V, d_V)$.

    As in (\ref{2.8}), there is a Hodge
    decomposition
    \begin{equation} \label{4.4}
      \Dom( \dbar^{0,\star}_{V,s}) =
      \Ker(\triangle_{V,s}^{0,\star}) \oplus \Image(\dbar_{V,s}) \oplus
      \Image(\dbar_{V,s}^*).
      \end{equation}
    Here the terms on the right-hand side of (\ref{4.4}) are the intersections
    of $\Dom( \dbar^{0,\star}_{V,s})$ with the corresponding terms in
    (\ref{2.8}). In particular,
    $\Ker(\triangle_{V,s}^{0,\star})$ and $\Image(\dbar_{V,s})$
    are the same, while the elements of $\Image(\dbar_{V,s}^*)$ now lie in an
    $H^1$-space.
Put
    \begin{equation} \label{4.5}
      {\mathcal I}_V = \dbar_{V,s}
\big|_{\Image(\dbar_{V,s}^*)}
      \: : \: \Image(\dbar_{V,s}^*) \rightarrow
    \Image(\dbar_{V,s}),
    \end{equation}
    an isomorphism.

    Define a linear map $m_V : \Dom( \dbar^{0,\star}_{V,s}) \oplus
    A_V^{\star-1} \rightarrow \Dom( \dbar^{0,\star}_{V,s}) \oplus
    A_V^{\star-1}$ by saying that if
\begin{equation} \label{4.6}
    (h, \omega_1, \omega_2, a) \in 
\Ker(\triangle_{V,s}^{0,\star}) \oplus \Image(\dbar_{V,s}) \oplus
\Image(\dbar_{V,s}^*) \oplus A_V^{\star -1}
\end{equation}
then
\begin{equation} \label{4.7}
m_V (h, \omega_1, \omega_2, a) =
(h, \omega_1, \omega_2, a + \gamma_V ({\mathcal I}^{-1}(\omega_1))).
\end{equation}
Its inverse is
given by
\begin{equation} \label{4.8}
m_V^{-1} (h, \omega_1, \omega_2, a) =
(h, \omega_1, \omega_2, a - \gamma_V ({\mathcal I}^{-1}(\omega_1))).
\end{equation}
    
\begin{proposition} \label{4.9}
  The linear maps $m_V$ and $m_V^{-1}$ are chain maps between
  $(T_V, d_V)$
  and
$(T_V, \widetilde{d}_V)$,
  i.e.
  $m_V \circ d_V = \widetilde{d}_V \circ m_V$ and
  $m_V^{-1} \circ \widetilde{d}_V = d_V \circ m^{-1}_V$.
  \end{proposition}
\begin{proof}
  We will check that $m_V \circ d_V = \widetilde{d}_V \circ m_V$;
  the proof that
  $m_V^{-1} \circ \widetilde{d}_V = d_V \circ m_V^{-1}$ is similar.

  Given $(h, \omega_1, \omega_2, a)$ as in (\ref{4.6}), we have

  \begin{align} \label{4.10}
    d_V (h, \omega_1, \omega_2, a) =
    & (0, \dbar_{V} \omega_2, 0, \gamma_V(h)), \\
    m_V(d_V (h, \omega_1, \omega_2, a)) =
    & (0, \dbar_{V} \omega_2, 0, \gamma_V(h) + \gamma_V(\omega_2)), \notag \\
    m_V (h, \omega_1, \omega_2, a) =
& (h, \omega_1, \omega_2, a + \gamma_V ({\mathcal I}^{-1}(\omega_1))), \notag \\
   \widetilde{d}_V( m_V (h, \omega_1, \omega_2, a)) =
& (0, \dbar_V \omega_2, 0, \gamma_V(h) + \gamma_V(\omega_2)).  \notag
    \end{align}
  This proves the proposition.
  \end{proof}

\begin{theorem} \label{4.11}
      The cohomology of
      $(T_V, {d}_V)$
      is isomorphic to $\HH^{*}(X; \underline{V})$.
\end{theorem}
    \begin{proof}
This follows from Propositions \ref{3.6} and \ref{4.9}.
      \end{proof}

    We can now
    reprove a result from \cite[Example 18.3.3 on p. 362]{Fulton (1998)}.

\begin{proposition} \label{4.12}
In terms of a resolution $\pi \: : \: M \rightarrow X$, we have
  \begin{align} \label{4.13}
\sum_{i=0}^n (-1)^i \dim(\HH^i(X; {\mathcal O}_X)) =
&  \int_M \Td(TM)
  - \dim(
  (\pi_* {\mathcal O_M}/{\mathcal O}_X)(X)) + \\
&  \sum_{i=1}^n (-1)^{i-1} 
    \dim( (R^i \pi_* {\mathcal O}_M )(X)). \notag
\end{align}
\end{proposition}
\begin{proof}
  Let $(T_1, d_1)$ denote the complex $(T_V, d_V)$
  when the vector bundle $V$ is the trivial bundle.
  From Theorem \ref{4.11}, the left-hand side of (\ref{4.13}) is the index of
  $d_1 + d_1^*$.
  We can deform the chain complex
  $(T_1, d_1)$ to make the differential 
  equal to $\dbar_{s} \oplus 0$ without changing the index.  The
  new index  is
  \begin{equation} \label{4.14}
    \sum_{i=0}^n (-1)^i \dim(\HH^i(\dbar_{s})) -
    \dim(({\mathcal O}_s/{\mathcal O}_X)(X)) + \sum_{i=1}^{n-1}
    (-1)^{i-1} \dim((\underline{\HH}^{0,i}(\overline{\partial}_{s}))(X)).
  \end{equation}
  From \cite{Pardon-Stern (1991)}, we have $\HH^i(\dbar_{s}) \cong
  \HH^{0,i}(M)$, so
  \begin{equation} \label{4.15}
    \sum_{i=0}^n (-1)^i \dim(\HH^i(\dbar_{s})) =
    \sum_{i=0}^n (-1)^i \dim(\HH^{0,i}(M)) = \int_M \Td(TM). 
    \end{equation}
    From \cite[Corollary 1.2]{Ruppenthal (2018)},
  we have ${\mathcal O}_s \cong \pi_* {\mathcal O}_M$ and
  $\underline{\HH}^{0,i}(\overline{\partial}_{s}) \cong
  R^i \pi_* {\mathcal O}_M$. The proposition follows.
\end{proof}

\begin{remark}
  We can write $\int_M \Td(TM) = \int_X \pi_* \Td(TM)$, where we are
  integrating
  a top-degree form on $X_{reg}$. It is not so clear
  what the relevant theory of characteristic classes on $X$ should be, for
  which this would be an example.
  We have in mind a Chern-Weil theory on $X_{reg}$ with control on how the
  forms behave near $X_{sing}$.
  We note that there is a rational
  homology class $\pi_* (PD [\Td(TM)])$ on $X$,
  where $PD [\Td(TM)] \in \HH_{even}(M; \Q)$ is the
  Poincar\'e dual of $[\Td(TM)] \in \HH^{even}(M; \Q)$, and
  if $X$ is connected then
  $\int_M \Td(TM)$ can be identified with the degree-zero component
  of $\pi_* (PD [\Td(TM)])$.
    \end{remark}

\section{K-homology class} \label{section5}

In this section we
prove Theorem \ref{1.5}.  We first show that if
$\pi \: : \: M \rightarrow X$ is a resolution of singularities, with
a simple normal crossing divisor, then
the
K-homology class $[\dbar_s + \dbar_s^*] \in \KK_0(X)$, from
Proposition \ref{2.1} with $V$ trivial,  equals the pushforward
$\pi_* [\dbar_M + \dbar_M^*]$. We then prove Theorem \ref{1.5}.

\begin{proposition} \label{sclass}
  Let $\pi \: : \: M \rightarrow X$ be a resolution of singularities, with
  $\pi^{-1}(X_{sing})$ being a simple normal crossing divisor.
  Then $[\dbar_s + \dbar_s^*] = \pi_* [\dbar_M + \dbar_M^*]$.
\end{proposition}
\begin{proof}
  The method of proof comes from \cite{Haskell (1987)}. Consider the
  following part 
of the K-homology exact sequence for the pair $(X, X_{sing})$:
\begin{equation}
  \KK_0(X_{sing}) \stackrel{\alpha}{\rightarrow}
  \KK_0(X) \stackrel{\beta}{\rightarrow}
  \KK_0(X, X_{sing}).
\end{equation}
\begin{lemma} \label{newlem}
  We have $\beta(  [\dbar_s + \dbar_s^*]) =
  \beta(\pi_* [\dbar_M + \dbar_M^*])$ in $\KK_0(X, X_{sing})$.
\end{lemma}
\begin{proof}
  Put $D = \pi^{-1}(X_{sing}) \subset M$. Since it has simple
  normal crossings, there will be a small regular neighborhood 
  of $D$ whose closure $C^\prime$ is
  homotopy equivalent to $D$. We can also assume that $C = \pi(C^\prime)$ is
  homotopy equivalent to $X_{sing}$ \cite[Theorem 2.10]{Milnor (1968)}.
  As $[\dbar_M + \dbar_M^*]$ is independent of the choice of Hermitian metric
  on $M$, we can choose a Hermitian metric on $M$ so that
$\pi$ restricts to an isometry from $M - C^\prime$ to $X-C$.

  Consider the commutative diagram
  \begin{equation}
\begin{array}{ccccccc}
  \KK_0(M) & \rightarrow & \KK_0(M, D) & \cong & \KK_0(M, C^\prime) & \cong & \KKK(C_0(M-C^\prime); \C) \\
  \pi_* \downarrow &  & \downarrow & & \downarrow & & \downarrow \\
  \KK_0(X) & \stackrel{\beta}{\rightarrow} & \KK_0(X, X_{sing}) & \cong &
  \KK_0(X, C) & \cong & \KKK(C_0(X-C); \C).
\end{array}
\end{equation}
  Starting with $[\dbar_M + \dbar_M^*] \in \KK_0(M)$ and going along the
  top row, its image in $\KKK(C_0(M-C^\prime); \C)$
  is the restriction of the analytic K-homology class, i.e. one
  only acts by functions that vanish on $C^\prime$.
  The right vertical arrow of the diagram is an isomorphism coming from the
  bijection between $M-C^\prime$ and $X-C$. By the commutativity of the
  diagram, we now know what
  $\beta(\pi_* [\dbar_M + \dbar_M^*])$ is as an element of
$\KKK(C_0(X-C); \C)$.
However, this is isomorphic to the
  restriction of $[\dbar_s + \dbar_s^*] \in \KK_0(X)$ to an element of
  $\KKK(C_0(X-C); \C)$ (since $\pi$ gives an isometry between
  $M-C^\prime$ and $X-C$). The latter
  restriction is the same as $\beta([\dbar_s + \dbar_s^*])$. 
This proves the lemma.
\end{proof}
To continue with the proof of Proposition \ref{sclass},
we know now that $[\dbar_s + \dbar_s^*]) -
\pi_* [\dbar_M + \dbar_M^*]$ lies in the kernel of $\beta$, and
so lies in the image of $\alpha$. For the purpose of the proof, we can
assume that $X$ is connected. Let
$a : \pt \rightarrow X$ be an arbitrary fixed embedding and let
$a_* \: : \: \KK_0(\pt) \rightarrow \KK_0(X)$ be the induced homomorphism.
The connectedness of $X$ implies that
$\Image(\alpha) = \Image(a_*)$. Let $b : X \rightarrow \pt$ be the unique
point map. Consider
$\pt \stackrel{a}{\rightarrow} X \stackrel{b}{\rightarrow} \pt$ and the
induced homomorphisms
$\KK_0(\pt) \stackrel{a_*}{\rightarrow} \KK_0(X) \stackrel{b_*}{\rightarrow}
\KK_0(\pt)$.  Then the map $b_*$ restricts to an isomorphism between
$\Image(a_*)$ and $\KK_0(\pt)$. Hence to prove the proposition, it suffices
to show that $b_* [\dbar_s + \dbar_s^*] = 
b_*(\pi_* [\dbar_M + \dbar_M^*])$ in $\KK_0(\pt) \cong \Z$.

Now $b_* [\dbar_s + \dbar_s^*]$ is the index of 
$\dbar_s + \dbar_s^*$, i.e.
$\sum_{i=0}^n (-1)^i \dim(\HH^i(\dbar_s))$, while
  $b_*(\pi_* [\dbar_M + \dbar_M^*])$ is the index of $\dbar_M + \dbar_M^*$,
  i.e.
  $\sum_{i=0}^n (-1)^i \dim(\HH^i(\dbar_M))$. From
    \cite{Pardon-Stern (1991)}, these are equal term-by-term.
    This proves the proposition.
  \end{proof}

We now prove Theorem \ref{1.5}. Suppose that $X$ is a connected projective
algebraic variety.
In terms of the resolution $\pi \: : \: M \rightarrow X$,
it was pointed out in
\cite[p. 104]{Baum-Fulton-MacPherson (1975)} that there is an identity in
$\KK_0(X)$: 
\begin{equation} \label{BFM}
[{\mathcal O}_X]_{BFM} - \pi_* [{\mathcal O}_M]_{BFM} = \sum_j n_j
[{\mathcal O}_{V_j}]_{BFM}.
\end{equation}
Here 
the $n_j$'s are certain integers and the
       $V_j$'s are irreducible subvarieties of
the singular locus of $X$. In our case of isolated singularities,
the $V_j$'s are just the points $x_j$ in $X_{sing}$.
As $[{\mathcal O}_M]_{BFM} = [\dbar_M + \dbar_M^*]$, 
  Proposition \ref{sclass} implies that
\begin{equation} \label{BFM2}
[{\mathcal O}_X]_{BFM} = [\dbar_s + \dbar_s^*] +  \sum_j n_j
[{\mathcal O}_{V_j}]_{BFM}.
\end{equation}  

  Let $(T_1, d_1)$ denote the complex $(T_V, d_V)$
  when the vector bundle $V$ is the trivial bundle.
  Let $[{\mathcal O}_X]_{an} \in \KK_0(X)$ be the K-homology class
  coming from the operator $d_1 + d_1^*$.
  We can deform the chain complex
  $(T_1, d_1)$ to make the differential 
  equal to $\dbar_{s} \oplus 0$ without changing the K-homology class
  arising from the complex.
  Then (\ref{BFM2}) implies that
  $[{\mathcal O}_X]_{an}$ and $[{\mathcal O}_X]_{BFM}$ have the same
  image in $\KK_0(X, X_{sing})$; c.f. the proof of Lemma \ref{newlem}.
  Let $b : X \rightarrow \pt$ be the unique point map.
 As in
 the proof of Proposition \ref{sclass}, 
 to conclude that $[{\mathcal O}_X]_{an} =
 [{\mathcal O}_X]_{BFM}$ in $\KK_0(X)$, it now suffices to show that
 $b_* [{\mathcal O}_X]_{an} =
 b_* [{\mathcal O}_X]_{BFM}$ in $\KK_0(\pt) \cong \Z$. Now
 $b_* [{\mathcal O}_X]_{an}$ is the index of $d_1 + d_1^*$ which, from
 Theorem \ref{4.11}, equals the arithmetic genus
 $\sum_{i=0}^n (-1)^i \dim(\HH^{i}(X; {\mathcal O}_X))$. On the other hand,
 from \cite[Section 3]{Baum-Fulton-MacPherson (1979)}, we also have
 $b_* [{\mathcal O}_X]_{BFM} = \sum_{i=0}^n (-1)^i \dim(\HH^{i}(X; {\mathcal O}_X))$. This proves the theorem.

\begin{remark}
  We mention some of the issues involved in extending the present paper to
  nonisolated singularities.  First, it seems to be open whether
  $\dbar_s + \dbar_s^*$ has compact resolvient, so the unbounded $\KKK$-formalism
  may not be applicable.  However, it is known that the unreduced cohomology of the $\dbar_s$-complex is finite dimensional, being isomorphic to the cohomology of a resolution \cite{Pardon-Stern (1991)}. Hence the $\dbar_s$-complex is
  Fredholm and one could use the bounded $\KKK$-description of K-homology,
  although it would be more cumbersome.
  
  We expect that Proposition \ref{sclass} still holds
  if $X$ has nonisolated
  singularities. It is known that
taking resolutions $\pi : M \rightarrow X$, the pushforward
$\pi_* [\overline{\partial}_M + \overline{\partial}_M^*] \in \KK_0(X)$
is independent of the
choice of resolution \cite{Hilsum (2018)}.

  One could ask for an extension of Theorem \ref{4.11} to the case of
  nonisolated singularities.  As an indication, one would expect that
  taking products of complex spaces would lead to tensor products of
  the cochain complexes.  In particular, suppose that $Z$ is a smooth
  Hermitian manifold and $X$ has isolated singular points.  Then the
  cochain complex for $Z \times X$ would have contributions from
  differential forms along the singular locus.

  In a related vein, in principle one can apply (\ref{BFM})
       inductively to get an expression for $[{\mathcal O}_X]_{BFM}$. 
\end{remark}

\end{document}